\documentclass[11pt]{amsart}

\usepackage{latexsym}
\usepackage{amsmath}
\usepackage{amsfonts}
\usepackage{amssymb}
\usepackage{enumitem}
\usepackage {tikz}
\usepackage{color}
\usetikzlibrary {positioning}
\usetikzlibrary{arrows,shapes}
\usepackage{color}
\usepackage{graphicx}
\usepackage{latexsym}
\usepackage{soul}

\usepackage[noend]{algorithmic} 
\usepackage{algorithm,caption}
\algsetup{indent=2em} 
 
\usepackage{tikz}
\usepackage{color}
\usepackage{placeins}
\usetikzlibrary{positioning}
\usetikzlibrary{arrows,shapes}

\newtheorem{theorem}{Theorem}[section]
\newtheorem{lemma}[theorem]{Lemma}

\newtheorem{proposition}[theorem]{Proposition}

\newtheorem{sublemma}{}[theorem]

\theoremstyle{definition}

\theoremstyle{remark}

\numberwithin{equation}{section}

\newcommand\Ga{\mathcal{G}^\mathrm{epex}}

\begin{document}

\title[Edge-apex graphs]{Edge-apexing in hereditary classes of graphs}


\author[Singh]{Jagdeep Singh}
\address{Department of Mathematics and Statistics\\
Binghamton University\\
Binghamton, New York}
\email{singhjagdeep070@gmail.com}

\author[Sivaraman]{Vaidy Sivaraman}
\address{Department of Mathematics and Statistics\\
Mississippi State University\\
Mississippi}
\email{vaidysivaraman@gmail.com}


\


\date{\today}

\begin{abstract}
A class $\mathcal{G}$ of graphs is called hereditary if it is closed under taking induced subgraphs. We denote by $\Ga$ the class of graphs that are at most one edge away from being in $\mathcal{G}$. We note that $\Ga$ is hereditary and prove that if a hereditary class $\mathcal{G}$ has finitely many forbidden induced subgraphs, then so does $\Ga$.

The hereditary class of cographs consists of all graphs $G$ that can be generated from $K_1$ using complementation and disjoint union. Cographs are precisely the graphs that do not have the $4$-vertex path as an induced subgraph. For the class of edge-apex cographs our main result bounds the order of such forbidden induced subgraphs by 8 and finds all of them by computer search.
\end{abstract}

\maketitle

\section{Introduction}
\label{intro}
We consider finite, simple, undirected graphs. An {\bf induced subgraph} of a graph $G$ is a graph $H$ that can be obtained from $G$ by a sequence of vertex deletions; $G-v$ is the graph obtained from $G$ by deleting a vertex $v$ of $G$. We write $G[T]$ for the subgraph if $T$ is the set of remaining vertices. We say that $G$ {\it contains} $H$ when $H$ is an induced subgraph of $G$. A class $\mathcal{G}$ of graphs is called {\bf hereditary} if it is closed under taking induced subgraphs.  
For a hereditary class $\mathcal{G}$, we call a graph $H$ a {\bf forbidden induced subgraph} for $\mathcal{G}$ if $H$ is not in $\mathcal{G}$ but every proper induced subgraph of $H$ is in $\mathcal{G}$. Thus, a graph $G$ is in $\mathcal{G}$ if and only if $G$ contains no forbidden induced subgraph for $\mathcal{G}$. 

For an edge $e$ of $G$, the graph $G-e$ is the graph obtained from $G$ by deleting $e$.
The {\bf edge-apex class, $\Ga$,} for a class $\mathcal{G}$ is the class of graphs $G$ such that $G$ is in $\mathcal{G}$, or $G$ has an edge $e$ such that $G-e$ is in $\mathcal{G}$. It is easy to see that if $\mathcal{G}$ is hereditary, then so is $\Ga$. We prove that if $\mathcal{G}$ has a finite set of forbidden induced subgraphs, then so does $\Ga$, and we apply this to the class of cographs.

A {\bf cograph} is a graph that can be generated from the single-vertex graph $K_1$ using the operations of complementation and disjoint union. {A graph $G$ is an {\bf edge-apex cograph} if $G$ is a cograph, or $G$ has an edge $e$ such that $G-e$ is a cograph.}  {In Theorem \ref{cograph_extension} we characterize edge-apex cographs by their forbidden induced subgraphs.} Cographs are also called $P_4$-free graphs due to the following characterization. 

\begin{theorem}[\cite{corneil}]
\label{cographs_characterisation}
A graph $G$ is a cograph if and only if $G$ does not contain the path $P_4$ on four vertices as an induced subgraph.
\end{theorem}

For a graph $G$, the graph $\overline{G}$ denotes the complement of $G$.  Note that $G$ is a cograph if and only if $\overline{G}$ is. We prove the following.

\begin{theorem}
\label{new1}
Let $\mathcal{G}$ be a hereditary class of graphs such that $\mathcal{G}$ has a finite number of forbidden induced subgraphs. Then the edge-apex class of $\mathcal{G}$ is hereditary and has a finite number of forbidden induced subgraphs. 
\end{theorem}

Its proof is in Section 2. In Section 3, we show that the number of vertices of a forbidden induced subgraph for the class of edge-apex cographs is at most eight. Following is the precise statement.

\begin{theorem}
\label{cograph_extension}
Let $G$ be a forbidden induced subgraph for the class of edge-apex cographs. Then $5 \leq |V(G)| \leq 8$, and $G$ is isomorphic to one of the graphs in Figures \ref{cograph_5}--\ref{cograph_8_3}.
\end{theorem}

The $5$-vertex such graphs are $C_5$, the $5$-cycle, and $C_5 + e$, the $5$-cycle plus a chord. The first graph in the second row of Figure \ref{cograph_6}, $K_3$ plus three leaves, is denoted by $K_3 \circ K_2$. We provide a readable proof that the forbidden induced subgraphs for the class of edge-apex cographs containing two vertex-disjoint $P_4$'s are precisely the graphs shown in Figure \ref{cograph_8_3}. Finding other such graphs led to a messy case analysis and the authors found it reasonable to list them using Sagemath.

\section{edge-apex exclusions}

We omit the proof of the following straightforward observation.

\begin{lemma}
\label{hereditary_apex}
For a hereditary class of graphs $\mathcal{G}$, the edge-apex class $\Ga$ is hereditary.
\end{lemma}

The proof of Theorem \ref{new1} is immediate from Lemma \ref{hereditary_apex} and the following proposition.

\begin{proposition}
\label{unique_FIS}
Let $\mathcal{G}$ be a hereditary class of graphs and let $c$ and $k$ denote the maximum number of vertices and edges in a forbidden induced subgraph for $\mathcal{G}$. If $G$ is a forbidden induced subgraph for $\Ga$, then $|V(G)| \leq$ max $\{2c, c + k(c-2)\}$. 
\end{proposition}

\begin{proof}
First note that if $G$ contains two edge-disjoint forbidden subgraphs $H_1$ and $H_2$ for $\mathcal{G}$, then $V(G) = V(H_1) \cup V(H_2)$. Suppose not and let $x$ be a vertex in $V(G)- (V(H_1) \cup V(H_2))$. Then $G-x$ is not in $\Ga$, a contradiction. It follows that $|V(G)| \leq 2c$.

Since $G$ is not in $\mathcal{G}$, the graph $G$ contains a forbidden induced subgraph $H$ for $\mathcal{G}$. Observe that for every edge $h$ of $H$, the graph $G-h$ is not in $\mathcal{G}$ so we have a subset $F_h$ of $V(G)$ such that $(G-h)[F_h]$ is a forbidden induced subgraph for $\mathcal{G}$. If there is a vertex $x$ in $V(G)- (V(H) \cup \bigcup_{h \in E(H)} F_h )$, then $G-x$ is not in $\Ga$, a contradiction. Therefore, $V(G) = V(H) \cup \bigcup_{h \in E(h)} F_h$.

Observe that $F_h$ contains both the endpoints of the edge $h$, else $F_h$ induces a forbidden induced subgraph for $\mathcal{G}$ in $G$ that is edge-disjoint from $H$ and thus $|V(G)| \leq 2c$. So for every $h$ in $E(H)$, we have $|V(H) \cap F_h| \geq 2$. The result now follows. 
\end{proof}

The bound can be slightly improved if a forbidden induced subgraph for $\Ga$ contains more than one forbidden induced subgraph for $\mathcal{G}$.

\begin{proposition}
\label{main_extension}
Let $\mathcal{G}$ be a hereditary class of graphs and let $G$ be a forbidden induced subgraph for $\Ga$ that has distinct subsets $S$ and $T$  of $V(G)$  such that $G[S]$ and $G[T]$ are forbidden induced subgraphs for $\mathcal{G}$. If $|S \cap T| = q$ and $c$ denotes the maximum number of vertices in a forbidden induced subgraph for $\mathcal{G}$, then $|V(G)| \leq c+(c-q)+k(c-2)$, where $k= |E(G[S]) \cap E(G[T])|.$
\end{proposition}

\begin{proof}
First we consider the possibility that $G[S]$ and $G[T]$ are edge-disjoint. In this case, we
have $V(G)= S \cup T$. Suppose not and let $x$ be a vertex in $V(G)- (S \cup T)$. Then $G-x$ is not in $\Ga$, a contradiction. It now follows that $|V(G)| \leq c+(c-q)$, verifying the proposition. Therefore, we may assume that for every choice of $S$ and $T$, $G[S]$ and $G[T]$ are not edge-disjoint. 

Let $E(G[S]) \cap E(G[T]) = \{e_1, \ldots, e_k\}$. Since for each $1 \leq i \leq k$, the graph $G-e_i$ is not in $\mathcal{G}$, we have a subset $F_i$ of $V(G)$ such that $(G-e_i)[F_i]$ induces a forbidden induced subgraph for $\mathcal{G}$. Observe that, if there is a vertex $x$ in $V(G)- (S \cup T \cup \bigcup_{i=1}^{k} F_i )$, then $G-x$ is not in $\Ga$, a contradiction. Therefore, $V(G) = S \cup T \cup \bigcup_{i=1}^{k} F_i$.

For $i$ in $\{1, \ldots, k\}$, if $F_i$ does not contain the endpoints of $e_i$, then $G[F_i]$ is a forbidden induced subgraph for $\mathcal{G}$. Since $G[S]$ and $G[F_i]$ are not edge-disjoint, we have $|F_i \cap S| \geq 2$ for each $i$ in $\{1,\ldots,k\}$. It now follows that $|V(G)| \leq c+(c-q)+k(c-2).$ 
\end{proof}

\section{Edge-apex Cograph Exclusions}

\begin{lemma}
\label{connected}
Let $G$ be a forbidden induced subgraph for the class of edge-apex cographs and assume that $G$ or $\overline{G}$ is disconnected. Then $G$ contains two vertex-disjoint $P_4$'s and $|V(G)|=8$.
\end{lemma}

\begin{proof}
First suppose that $G$ is disconnected. Assume that only one of the components, say $H$, of $G$ contains a $P_4$. Since $H$ is a proper induced subgraph of $G$, there is an edge $e$ of $H$ such that $H-e$ is a cograph. It follows that $G-e$ is a cograph, contradicting the minimality of $G$. Therefore more than one component of $G$ contains a $P_4$.

Next suppose that $\overline{G}$ is disconnected. If more than one component of $\overline{G}$ contains a $P_4$, then those $P_4$'s induce $P_4$'s in $G$ as well and the result follows. Suppose not, and let $\overline{H}$ be the only component of $\overline{G}$ that contains a $P_4$. It follows that the complement $H$ of $\overline{H}$ contains a $P_4$. Since the $P_4$'s in $G$ and $\overline{G}$ are in one-to-one correspondence, the $P_4$ in $H$ is the unique $P_4$ of $G$. Therefore there is an edge $e$ of $G$ such that $G-e$ is a cograph, a contradiction. Therefore $\overline{G}$ contains two vertex-disjoint $P_4$'s and so does $G$. 

Say $P$ and $P'$ are two such $P_4$'s in $G$. It follows that $V(G) = V(P) \cup V(P')$ and $|V(G)|=8$. 
\end{proof}

We omit the straightforward proof of the following result, which will be a useful tool in the proof of Proposition \ref{10_vtx_cographs} and Theorem \ref{cograph_extension}.

\begin{lemma}
\label{help}
For a graph $G$ and an edge $ab$ of $G$, if every $P_4$ of $G$ contains the edge $ab$ and $ab$ is contained in no cycle, then $G-ab$ is a cograph. 
\end{lemma}

\begin{proposition}
\label{10_vtx_cographs}
Let $G$ be a forbidden induced subgraph for the class of edge-apex cographs. Then $G$ contains more than one $P_4$.    
\end{proposition}

\begin{proof}
Suppose not, and let $abcd$ be the unique $P_4$ contained in $G$. Then $G-a$ is a cograph so $G-a$ or $\overline{G}-a$ is disconnected. By Lemma \ref{connected}, we may assume that both $G$ and $\overline{G}$ are connected. Observe that if $G-a$ is disconnected, then $\{b,c,d\}$ is in a single component of $G-a$. Let $v$ be a neighbour of $a$ in the other component. Then $vabc$ is an induced $P_4$ in $G$, a contradiction. Therefore $\overline{G}-a$ is disconnected. Note that $bdac$ is a $P_4$ in $\overline{G}$ so $\{b,d\}$ is in a single component of $\overline{G}-a$. Let $v$ be a neighbour of $a$ distinct from $c$ in the other component. Then $vadb$ is an induced $P_4$ in $\overline{G}$ and so induces a $P_4$ in $G$ other than $abcd$. It follows that $G-a$ has exactly two components and $v=c$. Note that if $c$ has a neighbour $w$ in $\overline{G}$, then $wcad$ is an induced $P_4$ in $\overline{G}$ so induces a $P_4$ in $G$, a contradiction. It follows that $c$ is a leaf in $\overline{G}$. Since $G-d$ is a cograph, by symmetry, $b$ is a leaf in $\overline{G}$. 

Next we show that $a$ is a leaf in $G$. Suppose that $t$ is a vertex in $G$ distinct from $b$ such that $at$ is an edge in $G$. Then $at$ is not an edge in $\overline{G}$. Since $\overline{G}$ is connected, there is an $at$-path $P$ in $\overline{G}$. Pick a shortest such path. Then $|P| \geq 2$. Note that $P$ along with the edge $ac$ contains a $P_4$ so induces a $P_4$ in $G$ other than $abcd$. It follows that $a$ is a leaf in $G$. Therefore, by Lemma \ref{help}, $G-ab$ has no $P_4$, a contradiction to the minimality of $G$. 
\end{proof}

\begin{proof}[Proof of Theorem \ref{cograph_extension}]
It is clear that $V(G) \geq 5$. By Proposition \ref{10_vtx_cographs}, we may assume that we have two distinct subsets $S$ and $T$ of $V(G)$ such that both $S$ and $T$ induce a $P_4$. Suppose that $|S \cap T| = q$ is minimal. Then by Proposition \ref{main_extension}, $V(G) \leq 4 + (4-q) + 2k$ where $k = |E(G[S]) \cap E(G[T])|$. Note that if $q \leq 1$, then $k=0$ so $V(G) \leq 8$. It remains to consider the cases when $q = 2$ or $q=3$. Observe that if $q = 2$, then $k$ is either $0$ or $1$. It is straightforward to check that $V(G) \leq 8$. Finally suppose that $q=3$. Since three vertices of a $P_4$ do not induce a $K_3$, it is clear that $k$ is in $\{0,1,2\}$. Note that $|V(G)| \leq 8$ unless $k=2$. Therefore assume that $k=2$ and $|V(G)|=9$. Since $G[S]$ and $G[T]$ have three common vertices and two common edges, the non-common vertex is not internal. Therefore we may assume that the two induced $P_4$'s in $G$ are $abcd$ and $abce$. Since $G-ab$ and $G-bc$ are not cographs, there are subsets $F_1$ and $F_2$ of $V(G)$ such that $(G-ab)[F_1]$ is a $P_4$ and $(G-bc)[F_2]$ is a $P_4$. Note that $V(G) = \{a,b,c,d,e\} \cup F_1 \cup F_2$. If, for $\{i,j\}=\{1,2\}$, $G[F_i]$ is a $P_4$, then by minimality of $q$, we have $|F_i \cap S| \geq 3$. It follows that $G[F_j]$ is not a $P_4$ otherwise $|F_j \cap S| \geq 3$ and so $|V(G)| \leq 7$, a contradiction. Note that $F_j$ contains the endpoints of the edge $ab$ of $G$ or the edge $ac$ of $G$, otherwise $G[F_j]$ is a $P_4$. Therefore $|F_j \cap S| \geq 2$ so $|V(G)| \leq 8$, again a contradiction. Thus we may assume that both $G[F_1]$ and $G[F_2]$ are not $P_4$'s. Since $\{a,b\} \subseteq F_1$ and $\{b,c\} \subseteq F_2$, and $|V(G)|=9$, we have $|F_i \cap \{a,b,c,d,e\}| = 2$ for $i$ in $\{1,2\}$. Suppose that $G[K]$ is a $P_4$ for a subset $K$ of $V(G)$ such that $G[K]$ does not contain the edge $ab$. Then $V(G)=S \cup T \cup K \cup F_2$. Since $|K \cap S| \geq 3$, we have $|V(G)| \leq 8$, a contradiction. Similarly, $G[K]$ contains the edge $bc$. Therefore by the minimality of $q$, we may assume that every induced $P_4$ of $G$ contains the edges $ab$ and $bc$.


If $G-b$ is not a cograph, then there exists an induced $P_4$ of $G$ that does not contain the edge $ab$, a contradiction. Therefore $G-b$ is a cograph. First suppose that $G-b$ is disconnected. Then $\{c,d,e\}$ is in a single component of $G-b$. Let $v$ be a neighbour of $b$ distinct from $a$ in the other component of $G-b$. Then $vbcd$ is an induced $P_4$ of $G$ that does not contain the edge $ab$, a contradiction. Thus $v=a$ and the edge $ab$ is in no cycle of $G$. By Lemma \ref{help}, it follows that $G-ab$ has no $P_4$, a contradiction. 

Therefore $\overline{G}-b$ is disconnected. Then $\{a,c,d,e\}$ is in a single component of $\overline{G}-b$. Let $v$ be a neighbour of $b$ in the other component. Then $vbda$ is an induced $P_4$ in $\overline{G}$ so $\{a,b,d,v\}$ induces a $P_4$ in $G$ that does not contain the edge $bc$, a contradiction. Therefore $|V(G)| \leq 8$.
\end{proof}

We omit the proof of the following straightforward result. 

\begin{lemma}
\label{lemmaA}
Let $G$ be a forbidden induced subgraph for edge-apex cographs. Then no proper induced subgraph of $G$ contains two edge disjoint $P_4$'s.
\end{lemma}

In the rest of the section, we consider the case when a forbidden induced subgraph $G$ for edge-apex cographs contains two vertex-disjoint $P_4$'s. We determine all such graphs.

\begin{proposition}
\label{8_vtx}
Let $G$ be a forbidden induced subgraphs for the class of edge-apex cographs such that $G$ contains two vertex-disjoint $P_4$'s. Then $G$ is isomorphic to one of the graphs in Figure \ref{cograph_8_3}.
\end{proposition}

\begin{proof}
It can be easily checked that the graphs in Figure \ref{cograph_8_3} are forbidden induced subgraphs for edge-apex cographs. Let $abcd$ and $1234$ be two vertex-disjoint induced $P_4$'s in $G$. First we consider the case when $a1$ is an edge of $G$. 

\begin{sublemma}
\label{sublemma1}
$a$ is adjacent to all of $\{1,2,3,4\}$.
\end{sublemma}

If $a2$ and $a3$ are non-edges of $G$, then $abcd$ and $a123$ are two induced edge-disjoint $P_4$'s in $G-4$, a contradiction to Lemma \ref{lemmaA}. Suppose that $a2$ is not an edge of $G$. It follows that $a3$ is an edge. If $a4$ is an edge then $\{a,1,2,3,4\}$ induce a $C_5 + e$, a contradiction. Therefore $a4$ is not an edge of $G$. Observe that $43a1$ is an induced $P_4$, contradicting Lemma \ref{lemmaA}. Therefore $a2$ is an edge of $G$. Then $a3$ or $a4$ is an edge of $G$, otherwise $a234$ and $abcd$ are edge-disjoint $P_4$'s, a contradiction. If $a4$ is not an edge, then $a3$ is an edge and $43a1$ is an induced $P_4$ edge-disjoint from $abcd$, a contradiction to Lemma \ref{lemmaA}. Therefore $a4$ is an edge of $G$. Note that $a3$ is also an edge else we get an induced $C_5 + e$ in $G$. Therefore $a$ is adjacent to all of $\{1,2,3,4\}$ and \ref{sublemma1} holds.

By symmetry and \ref{sublemma1}, it follows that $1$ is adjacent to all of $\{a,b,c,d\}$. Since $d1$ and $a4$ are edges of $G$, symmetry and \ref{sublemma1} implies that $d$ is adjacent to all of $\{1,2,3,4\}$ and $4$ is adjacent to all of $\{a,b,c,d\}$ in $G$.

\begin{sublemma}
\label{sublemma2}
$b$ is adjacent to both $2$ and $3$.
\end{sublemma}

Observe that $b$ is adjacent to at least one of $\{2,3\}$ else $b123$ is an induced $P_4$ in $G$, a contradiction to Lemma \ref{lemmaA}. Suppose $b3$ is not an edge. Then $b2$ is an edge and $\{b,1,2,3,4\}$ induces a $C_5 + e$ in $G$, a contradiction. Therefore $b3$ is an edge. If $b2$ is not an edge, then we again get an induced $C_5 + e$. It follows that $b$ is adjacent to both $2$ and $3$. This proves \ref{sublemma2}. 

By symmetry and \ref{sublemma2}, it follows that $c$ is adjacent to both $2$ and $3$. Therefore $G$ is isomorphic to the graph shown in Figure \ref{cograph_8_3}(c).

Now we consider the case when $a1$ is not an edge of $G$. By symmetry, we may also assume that none of $\{a4, d1, d4\}$ is an edge of $G$. 

\begin{sublemma}
\label{sublemma3}
At least one of $\{a,d,1,4\}$ is a leaf.
\end{sublemma}

Suppose not. Since $a$ is not a leaf, it is adjacent to at least one of $\{2,3\}$. If $a$ has exactly one neighbour in $\{2,3\}$, then we get two induced edge-disjoint $P_4$'s in a proper induced subgraph of $G$. This is a contradiction to Lemma \ref{lemmaA}. Therefore $a$ is adjacent to both $2$ and $3$. Similarly, as $d, 1,$ and $4$ are not leaves, all of $\{d2,d3,1b,1c,4b,4c\}$ are edges of $G$. If $c2$ is not an edge of $G$, then $a21c$ is an induced $P_4$ of $G$, a contradiction to Lemma \ref{lemmaA}. Therefore $c2$ is an edge of $G$. It now follows that $4c2a$ is an induced $P_4$ of $G$ edge-disjoint from $abcd$, a contradiction. Therefore our result holds.

By symmetry, we may assume that $a$ is a leaf. 

\begin{sublemma}
\label{sublemma4}
$d$ is a leaf.
\end{sublemma}

Suppose that $d$ is not a leaf. Then $d$ is adjacent to both $2$ and $3$ else we get two induced edge-disjoint $P_4$'s having a common vertex, a contradiction to Lemma \ref{lemmaA}. Note that if $b$ is adjacent to $3$, then $d3ba$ is an induced $P_4$ again a contradiction to Lemma \ref{lemmaA}. So $b$ is not adjacent to $3$. Similarly $b2$ is not an edge of $G$. If $b4$ is an edge, then $b432$ is an induced $P_4$ of $G$, a contradiction to Lemma \ref{lemmaA}. In the same way $b1$ is not an edge of $G$. Therefore $b$ has no neighbour in $\{1,2,3,4\}$. If $c$ has a neighbour $v$ in $\{1,2,3,4\}$, then $abcv$ is an induced $P_4$, a contradiction to Lemma \ref{lemmaA}. So $c$ has no neighbour in $\{1,2,3,4\}$. It follows that $bcd2$ is a $P_4$ edge-disjoint from $1234$ contradicting Lemma \ref{lemmaA}. Therefore \ref{sublemma4} holds.

Suppose there is no edge between $\{b,c\}$ and $\{1,2,3,4\}$. Then $G$ is isomorphic to the graph in Figure \ref{cograph_8_3}(a). So we may assume that there is an edge between $\{b,c\}$ and $\{1,2,3,4\}$. Suppose $v$ is a neighbour of $b$ in $\{1,2,3,4\}$ such that $cv$ is not an edge of $G$. Then $dcbv$ is an induced $P_4$ of $G$ contradicting Lemma \ref{lemmaA}. Therefore every neighbour of $b$ in $\{1,2,3,4\}$ is also a neighbour of $c$. For $v$ in $\{1,2,3,4\}$, if $cv$ is an edge of $G$ but $bv$ is not an edge, then $abcv$ is an induced $P_4$, a contradiction to Lemma \ref{lemmaA}. Therefore $b$ and $c$ have the same neighbours in $\{1,2,3,4\}$. Let $v$ be a neighbour of $b$ and $c$ in $\{1,2,3,4\}$. 

\begin{sublemma}
\label{sublemma5}
A neighbour $w$ of $v$ in $\{1,2,3,4\}$ is adjacent to $b$ and $c$. 
\end{sublemma}

If not, then the graph induced on $\{a,b,c,d,v,w\}$ is $K_3 \circ K_2$, a contradiction. Therefore \ref{sublemma5} holds.

By \ref{sublemma5}, it now follows that $G$ is isomorphic to the graph in Figure \ref{cograph_8_3}(b). 
\end{proof}

We implemented the following algorithm on graphs with $5,6,7,$ and $8$ vertices using SageMath~\cite{sage} and list all forbidden induced subgraphs for edge-apex cographs in Figures \ref{cograph_5}, \ref{cograph_6}, \ref{cograph_7}, and \ref{cograph_8_3}. The graphs were drawn using SageMath.

\begin{algorithm}
\renewcommand{\thealgorithm}{}
\label{pseudocode2}
\caption{Finding forbidden induced subgraphs for edge-apex cographs}

\begin{algorithmic}
\REQUIRE $n = 5,6,7$ or $8$.
\STATE Set FinalList $\leftarrow \emptyset$
\STATE Generate all graphs of order $n$ using nauty geng \cite{nauty} and store in an iterator $L$

\FOR{$g$ in $L$ such that $g$ is not a cograph}
    \STATE Set $i \leftarrow 0$, $j \leftarrow 0$

    \FOR{$e$ in $E(g)$}
        \STATE $h = g - e$
        \IF{$h$ is not a cograph}
            \STATE $i \leftarrow i+1$
        \ENDIF
    \ENDFOR

    \FOR{$v$ in $V(g)$}
        \STATE $k = g - v$
        \IF{$k$ is a cograph or $k-e$ is cograph for some edge $e$ of $k$}
            \STATE $j \leftarrow j+1$
        \ENDIF
    \ENDFOR

\IF{$i$ equals $|E(g)|$ and $j$ equals $|V(g)|$}

\STATE Add $g$ to FinalList

\ENDIF

\ENDFOR

\end{algorithmic}
\end{algorithm}

\begin{figure}[htbp]
\centering

\begin{minipage}{.17\linewidth}
  \includegraphics[scale=0.32]{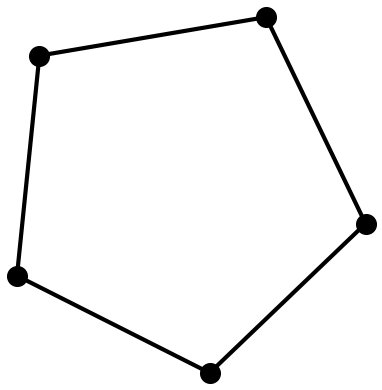}
\end{minipage} \hspace{.1\linewidth}
\begin{minipage}{.17\linewidth}
  \includegraphics[scale=0.32]{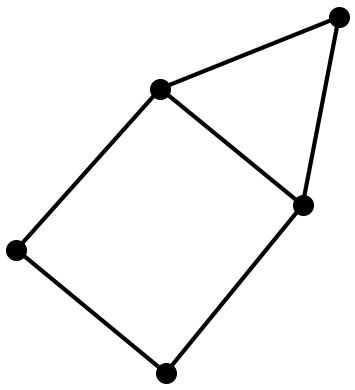}
\end{minipage} \hspace{.01\linewidth}

\caption{ The $5$-vertex forbidden induced subgraphs for edge-apex cographs.}

\label{cograph_5}

\end{figure}

\hspace{15mm}

\begin{figure}[htbp]
\centering
\begin{minipage}{.17\linewidth}
  \includegraphics[scale=0.27]{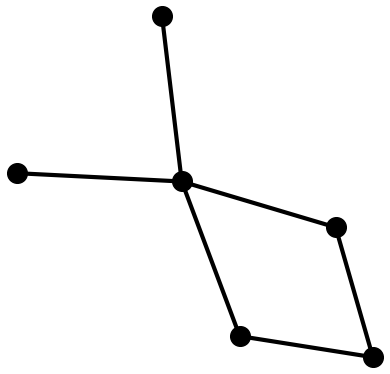}
\end{minipage}
\hspace{.05\linewidth}
\begin{minipage}{.17\linewidth}
\includegraphics[scale=0.27]{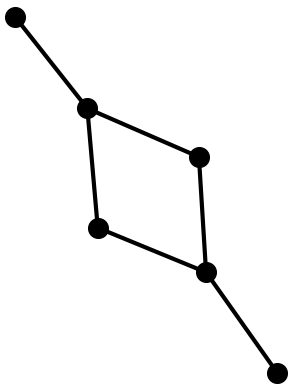}
\end{minipage}
\hspace{.05\linewidth}
\begin{minipage}{.17\linewidth}
\includegraphics[scale=0.27]{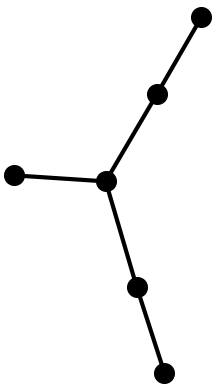}
\end{minipage}
\begin{minipage}{.17\linewidth}
  \includegraphics[scale=0.27]{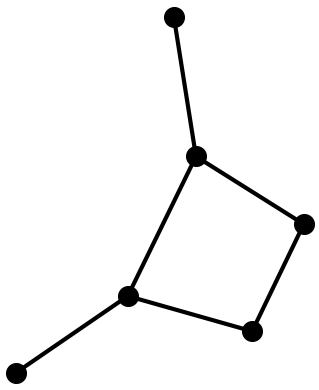}
\end{minipage}
\hspace{.05\linewidth}

\begin{minipage}{.17\linewidth}
  \includegraphics[scale=0.27]{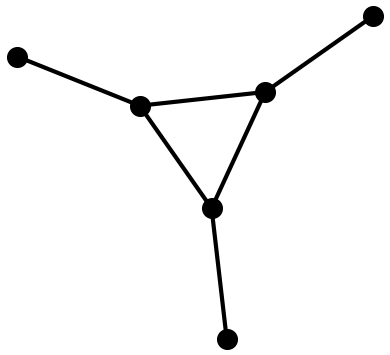}
\end{minipage} \hspace{.05\linewidth}
\begin{minipage}{.17\linewidth}
  \includegraphics[scale=0.27]{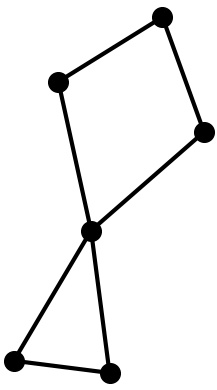}
\end{minipage}
\hspace{.05\linewidth}
\begin{minipage}{.17\linewidth}
\includegraphics[scale=0.27]{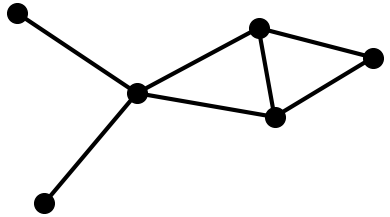}
\end{minipage}
\hspace{.05\linewidth}
\begin{minipage}{.17\linewidth}
\includegraphics[scale=0.27]{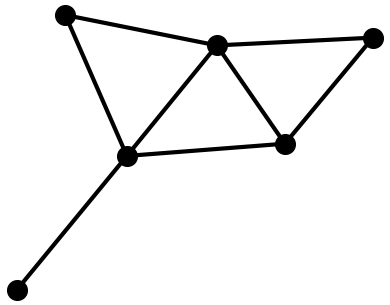}
\end{minipage}

\begin{minipage}{.17\linewidth}
  \includegraphics[scale=0.27]{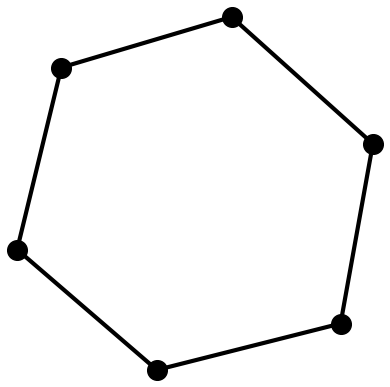}
\end{minipage}
\hspace{.07\linewidth}
\begin{minipage}{.17\linewidth}
  \includegraphics[scale=0.27]{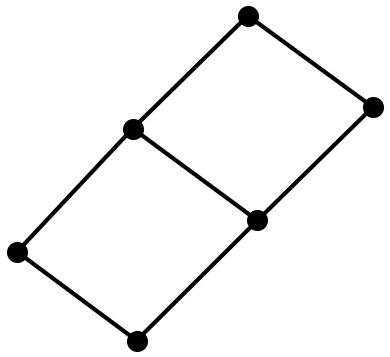}
\end{minipage} \hspace{.10\linewidth}
\begin{minipage}{.17\linewidth}
  \includegraphics[scale=0.27]{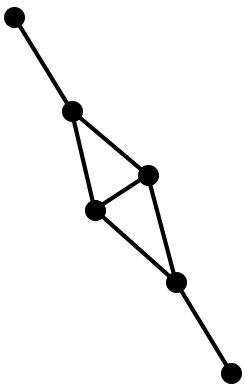}
\end{minipage} \hspace{.05\linewidth}
\begin{minipage}{.17\linewidth}
\includegraphics[scale=0.27]{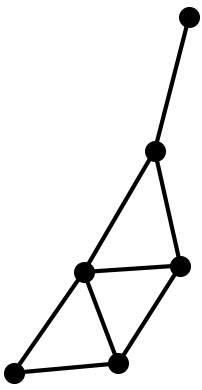}
\end{minipage}
\hspace{.05\linewidth}

\begin{minipage}{.17\linewidth}
\includegraphics[scale=0.27]{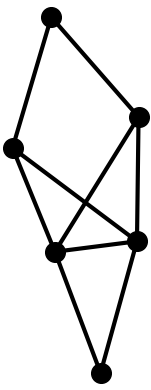}
\end{minipage}\hspace{.05\linewidth}
\begin{minipage}{.17\linewidth}
  \includegraphics[scale=0.27]{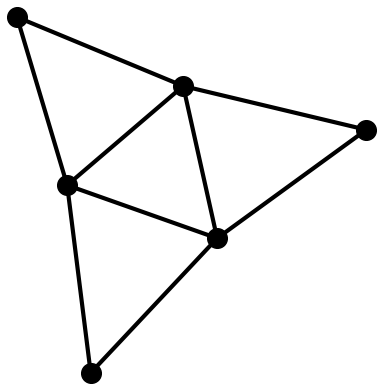}
\end{minipage}
\hspace{.05\linewidth}
\begin{minipage}{.17\linewidth}
  \includegraphics[scale=0.27]{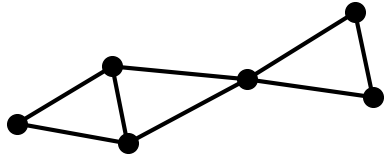}
\end{minipage} \hspace{.10\linewidth}
\begin{minipage}{.17\linewidth}
  \includegraphics[scale=0.27]{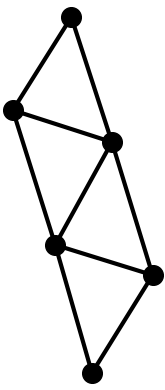}
\end{minipage}
\hspace{.05\linewidth}

\begin{minipage}{.17\linewidth}
\includegraphics[scale=0.27]{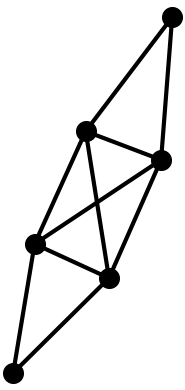}
\end{minipage}
\hspace{.05\linewidth}
\begin{minipage}{.17\linewidth}
\includegraphics[scale=0.27]{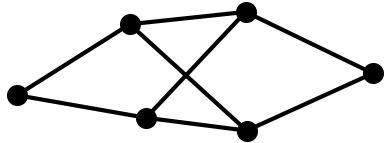}
\end{minipage}
\hspace{.01\linewidth}

\caption{The $6$-vertex forbidden induced subgraphs for edge-apex cographs.}

\label{cograph_6}

\end{figure}

\begin{figure}[htbp]
\centering

\begin{minipage}{.17\linewidth}
\includegraphics[scale=0.25]{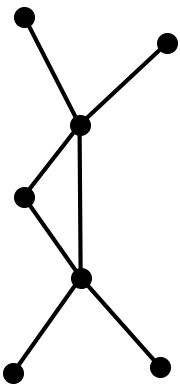}
\end{minipage}
\begin{minipage}{.17\linewidth}
  \includegraphics[scale=0.25]{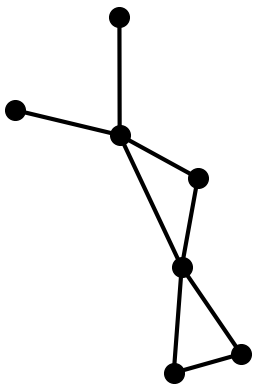}
\end{minipage}
\hspace{.01\linewidth}
\begin{minipage}{.17\linewidth}
  \includegraphics[scale=0.25]{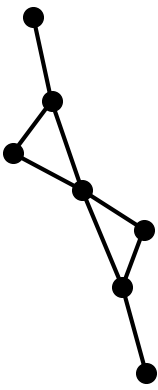}
\end{minipage} \hspace{.01\linewidth}

\begin{minipage}{.17\linewidth}
  \includegraphics[scale=0.25]{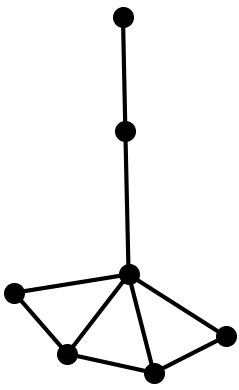}
\end{minipage}
\hspace{.01\linewidth}
\begin{minipage}{.17\linewidth}
\includegraphics[scale=0.25]{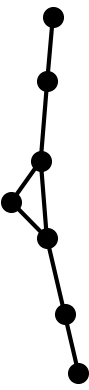}
\end{minipage}
\hspace{.01\linewidth}
\begin{minipage}{.17\linewidth}
\includegraphics[scale=0.25]{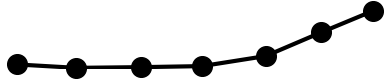}
\end{minipage}\hspace{.01\linewidth}

\begin{minipage}{.17\linewidth}
  \includegraphics[scale=0.25]{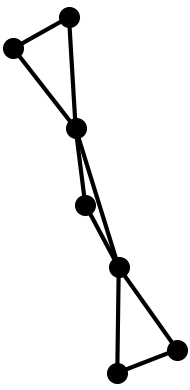}
\end{minipage}
\hspace{.01\linewidth}
\begin{minipage}{.17\linewidth}
  \includegraphics[scale=0.25]{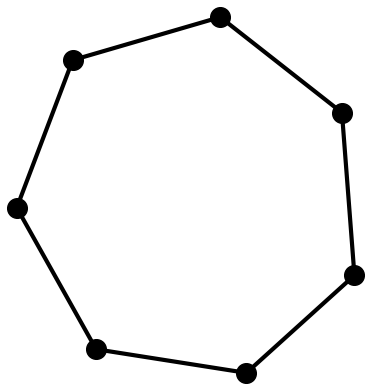}
\end{minipage} \hspace{.1\linewidth}
\begin{minipage}{.17\linewidth}
  \includegraphics[scale=0.25]{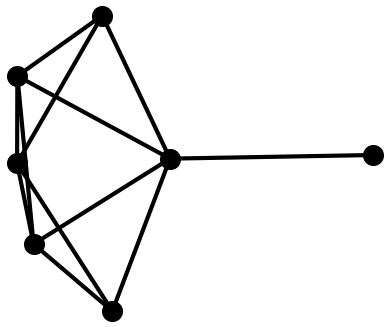}
\end{minipage}
\hspace{.01\linewidth}

\caption{The $7$-vertex forbidden induced subgraphs for edge-apex cographs.}

\label{cograph_7}

\end{figure}

\begin{figure}[htbp]
\begin{minipage}{.17\linewidth}
\includegraphics[scale=0.32]{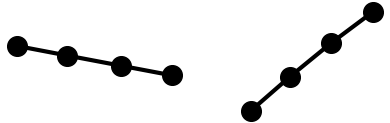} \\ (a)
\end{minipage}
\qquad\qquad\qquad
\begin{minipage}{.17\linewidth}
\includegraphics[scale=0.32]{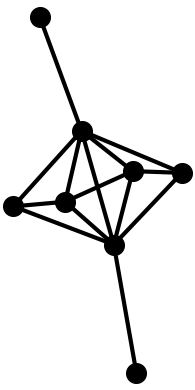} \\ (b)
\end{minipage} \hspace{.07\linewidth}
\quad
\begin{minipage}{.17\linewidth}
  \includegraphics[scale=0.32]{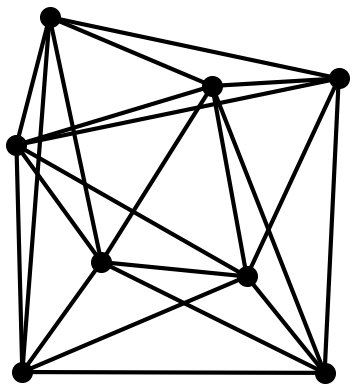} \\ (c)
\end{minipage}
\hspace{.01\linewidth}
\caption{The three $8$-vertex forbidden induced subgraphs for edge-apex cographs.}

\label{cograph_8_3}

\end{figure}

\FloatBarrier

\section*{Acknowledgement}

The authors thank Thomas Zaslavsky for helpful discussions. V. Sivaraman was partially supported by Simons Foundation Travel Support for Mathematicians (Award No. 855469).

\end{document}